\documentclass[12pt]{article}
\usepackage[left=2.4cm,right=2.4cm,top=2.4cm,bottom=2.4cm]{geometry}
\usepackage{amsfonts}
\usepackage{titlesec}
\usepackage{cite,color,xcolor}
\usepackage{amsmath}
\usepackage{amsthm}
\usepackage{amssymb}
\usepackage{times}
\usepackage{bm}
\usepackage[colorlinks,citecolor=blue,urlcolor=blue]{hyperref}
\usepackage[utf8]{inputenc}
\usepackage{mathtools}
\usepackage{setspace}
\usepackage[numbers]{natbib}
\usepackage{cases}
\usepackage{fancyhdr}

\newtheorem{theorem}{Theorem}[section]

\newtheorem{lemma}{Lemma}[section]
\newtheorem{proposition}{Proposition}[section]

\newtheorem{remark}{Remark}[section]

\usepackage{txfonts}

\newcommand{\bal}{\begin{align}}
\newcommand{\bbal}{\begin{align*}}
\newcommand{\beq}{\begin{equation}}
\newcommand{\eeq}{\end{equation}}
\newcommand{\bca}{\begin{cases}}
\newcommand{\eca}{\end{cases}}
\def\div{\mathord{{\rm div}}}
\newcommand{\pa}{\partial}
\newcommand{\fr}{\frac}

\newcommand{\dd}{\mathrm{d}}

\newcommand{\R}{\mathbb{R}}

\newcommand{\les}{\lesssim}

\newcommand{\F}{\dot{F}}

\begin{document}
\bibliographystyle{plain}
\title{On the well-posedness and non-uniform continuous dependence for the Novikov equation  in the
Triebel-Lizorkin spaces}

\author{Jinlu Li$^{1}$, Yanghai Yu$^{2,}$\footnote{E-mail: lijinlu@gnnu.edu.cn; yuyanghai214@sina.com(Corresponding author); mathzwp2010@163.com} and Weipeng Zhu$^{3}$\\
\small $^1$ School of Mathematics and Computer Sciences, Gannan Normal University, Ganzhou 341000, China\\
\small $^2$ School of Mathematics and Statistics, Anhui Normal University, Wuhu 241002, China\\
\small $^3$ School of Mathematics and Big Data, Foshan University, Foshan, Guangdong 528000, China}

\date{\today}

\maketitle\noindent{\hrulefill}

{\bf Abstract:} In this paper we study the Cauchy problem of the Novikov equation in $\R$ for initial data belonging to the Triebel-Lizorkin spaces, i.e, $u_0\in F^{s}_{p,r}$ with $1< p, r<\infty$ and $s>\max\{\frac32,1+\frac1p\}$. We prove local-in-time unique existence of solution to the Novikov equation in $F^{s}_{p,r}$. Furthermore, we obtain that the data-to-solution
of this equation is continuous but not uniformly continuous in the same space.


\vskip0mm\noindent{\hrulefill}

\thispagestyle{empty}
\section{Introduction}
\quad
In this paper, we consider the Cauchy problem of the Novikov equation
\begin{align}\label{N}
\begin{cases}
(1-\partial_x^2) \partial_t u+4 u^2 \partial_x u-3 u \partial_x u \partial_x^2 u-u^2 \partial_x^3 u=0, \\
u(x, 0)=u_0(x), \quad x \in \mathbb{R},
\end{cases}
\end{align}
and study its well-posedness and dependence on initial data in the Triebel-Lizorkin spaces. This equation has been discovered by Novikov \cite{VN} as a new integrable equation with cubic nonlinearities which can be thought as a generalization of the Camassa-Holm $(\mathrm{CH})$ equation. In \cite{VN}, Novikov investigated the question of integrability for Camassa--Holm type equations of the form
$$
(1-\partial_x^2) u_t=P\left(u, u_x, u_{x x}, u_{x x x}, \cdots\right),
$$
where $P$ is a polynomial of $u$ and its $x$-derivatives. Using as test for integrability the existence of an infinite hierarchy of
quasi-local higher symmetries, he produced about 20 integrable equations with quadratic nonlinearities that include the $\mathrm{CH}$ equation
\bal\label{ch}
(1-\partial_x^2) u_t=u u_{x x x}+2 u_x u_{x x}-3 u u_x
\end{align}
and the Degasperis-Procesi $(\mathrm{DP})$ equation
\bal\label{dp}
(1-\partial_x^2)  u_t=u u_{x x x}+3 u_x u_{x x}-4 u u_x .
\end{align}
Moreover, he produced about 10 integrable equations with cubic nonlinearities that include the  Novikov equation (NE)
\bal\label{ne}
(1-\partial_x^2)  u_t=u^2 u_{x x x}+3 u u_x u_{x x}-4 u^2 u_x.
\end{align}

The Camassa--Holm equation was originally derived as a bi-Hamiltonian system by Fokas and Fuchssteiner \cite{Fokas1981} in the context of the KdV model and gained prominence after Camassa--Holm \cite{Camassa1993} independently re-derived
it from the Euler equations of hydrodynamics using asymptotic expansions. \eqref{ch} is completely integrable \cite{Camassa1993,Constantin-P} with a bi-Hamiltonian structure \cite{Constantin-E,Fokas1981} and infinitely many conservation laws \cite{Camassa1993,Fokas1981}. Also, it admits exact peaked soliton solutions (peakons) of the form $ce^{-|x-ct|}$ with $c>0$, which are orbitally stable \cite{Constantin.Strauss} and models wave breaking (i.e., the solution remains bounded, while its slope becomes unbounded in finite time \cite{Constantin,Escher2,Escher3}).
The Degasperis-Procesi equation with a bi-Hamiltonian structure is integrable \cite{DP} and has traveling wave solutions \cite{Lenells}. Although DP is similar to CH  in several aspects, these two equations are truly
different. One of the novel features of  DP different from  CH is that it has not only peakon solutions \cite{DP} and periodic peakon solutions \cite{YinJFA}, but also shock peakons \cite{Lundmark2007} and the periodic shock waves \cite{Escher}.

For the Novikov equation, Hone-Wang \cite{Home2008} derived the Lax pair  which is given by
\bbal
\left(\begin{array}{l}\psi_1 \\ \psi_2 \\ \psi_3\end{array}\right)_x=U(m, \lambda)\left(\begin{array}{l}\psi_1 \\ \psi_2 \\ \psi_3\end{array}\right), \quad\left(\begin{array}{l}\psi_1 \\ \psi_2 \\ \psi_3\end{array}\right)_t=V(m, u, \lambda)\left(\begin{array}{l}\psi_1 \\ \psi_2 \\ \psi_3\end{array}\right),
\end{align*}
where $m=u-u_{x x}$ and the matrices $U$ and $V$ are defined by
\bbal
U(m, \lambda)=\left(\begin{array}{ccc}
0 & \lambda m & 1 \\
0 & 0 & \lambda m \\
1 & 0 & 0
\end{array}\right)
\end{align*}
and
\bbal V(m, u, \lambda)=\left(\begin{array}{lrr}
\frac{1}{3 \lambda^2}-u u_x & \frac{u_x}{\lambda}-\lambda m u^2 & u_x^2 \\
\frac{u}{\lambda} & -\frac{2}{3 \lambda^2} & -\frac{u_x}{\lambda}-\lambda m u^2 \\
-u^2 & \frac{u}{\lambda} & \frac{1}{3 \lambda^2}+u u_x
\end{array}\right) .
\end{align*}
NE possesses peakon traveling wave solutions \cite{HM,HLS,GH}, which on the real line are given by the formula
$
u(x, t)=\pm \sqrt{c} e^{-|x-c t|}
$
where $c>0$ is the wave speed. In fact, NE admits multi-peakon traveling wave solutions on both the line and the circle. More precisely, on the line the $n$-peakon
$$
u(x, t)=\sum_{j=1}^n p_j(t) e^{-\left|x-q_j(t)\right|}
$$
is a solution to $\mathrm{NE}$ if and only if the positions $\left(q_1, \ldots, q_n\right)$ and the momenta $\left(p_1, \ldots, p_n\right)$ satisfy the following system of $2 n$ differential equations:
$$
\left\{\begin{aligned}
\frac{\dd q_j}{\dd t} &=u^2(q_j), \\
\frac{\dd p_j}{\dd t} &=-u(q_j) u_x(q_j) p_j .
\end{aligned}\right.
$$
Furthermore, it has infinitely many conserved quantities. Like CH, the most important quantity conserved by a solution $u$ to NE is its $H^1$-norm
$$
\|u\|_{H^1}^2 =\int_{\mathbb{R}}\big[u^2+u_x^2\big] \dd x.
$$

We say that the Cauchy problem \eqref{dp} is locally well-posed
in a Banach space $X$ if the following three conditions hold
\begin{enumerate}
  \item (Local existence)\; For any initial data $u_0\in X$, there exists a short time $T = T(u_0) > 0$ and a solution $\mathbf{S}_{t}(u_0)\in\mathcal{C}([0,T),X)$ to the Cauchy problem \eqref{dp};
  \item (Uniqueness)\; This solution $\mathbf{S}_{t}(u_0)$ is unique in the space $\mathcal{C}([0,T),X)$;
  \item (Continuous Dependence)\; The data-to-solution map $u_0 \mapsto \mathbf{S}_{t}(u_0)$ is continuous in the following
sense:  for any $T_1 < T$ and $\varepsilon > 0$, there exists $ \delta> 0$, such that if $\|u_0-\widetilde{u}_0\|_{X}\leq \delta$,
then $\mathbf{S}_{t}(\widetilde{u}_0)$ exists up to $T_1$ and
$$\|\mathbf{S}_{t}(u_0)-\mathbf{S}_{t}(\widetilde{u}_0)\|_{\mathcal{C}([0,T),X)}\leq \varepsilon.$$
\end{enumerate}

The well-posedness of the Camassa--Holm type equations has been widely investigated during the past 20 years.
The local well-posedness for the Cauchy problem of CH \cite{LO,GB,Dan2,L16} and NE \cite{HH,HH0,N,Wu1,Wu2,Yan1,Yan2,Zhou} in Sobolev and Besov spaces $B_{p, r}^s(\mathbb{R})$ with  $s>\max\{1+1/{p}, 3/{2}\}$ and $(p,r)\in[1,\infty]\times[1,\infty)$ has been established.
In our recent papers \cite{Li22,Li22-jee}, we established the ill-posedness for CH in $B^s_{p,\infty}(\mathbb{R})$  with $p\in[1,\infty]$ by proving the solution map  starting from $u_0$ is discontinuous at $t = 0$ in the metric of $B^s_{p,\infty}(\mathbb{R})$.
Guo-Liu-Molinet-Yin \cite{Guo2019} established the ill-posedness for
the Camassa--Holm type equations in $B_{p,r}^{1+1/p}(\mathbb{R})$ with $(p,r)\in[1,\infty]\times(1,\infty]$ by proving the norm inflation, which implies that $B_{p, 1}^{1+1/p}$ is the critical Besov space for both CH and NE.  Ye-Yin-Guo \cite{Ye} obtained the local well-posedness for the Camassa--Holm type equation in critical Besov spaces $B^{1+1/p}_{p,1}(\R)$ with $p\in[1,\infty)$.  Yang \cite{Yang} obtained well-posedness of CH in the Triebel-Lizorkin spaces $F^{s}_{p,r}$ with $1< p, r<\infty, s>\max\{3/2,1+1/p\}$ and presented a blow-up criterion of solutions. Chae in \cite{C1,C2} and Chen-Miao-Zhang in \cite{CMZ} studied the existence and uniqueness of the Euler and MHD equations in general Triebel–Lizorkin spaces, respectively. To the best of our knowledge, whether NE is well-posed or not in $F^{s}_{p,r}$ is still an open problem.

Setting $\Lambda^{-2}=(1-\pa^2_x)^{-1}$, then $\Lambda^{-2}f=G*f$ where $G(x)=\fr12e^{-|x|}$ is the kernel of the operator $\Lambda^{-2}$. We can transform the Novikov equation into the following transport type equation
\begin{equation}\label{5}
\begin{cases}
u_t+u^2u_x=\mathbf{P}(u)=\mathbf{P}_1(u)+\mathbf{P}_2(u),\\
u(x,t=0)=u_0(x),
\end{cases}
\end{equation}
where
\begin{equation}\label{6}
\mathbf{P}_1(u)=-\frac12\Lambda^{-2}u_x^3\quad\text{and}\quad \mathbf{P}_2(u)=-\pa_x\Lambda^{-2}\left(\frac32uu^2_x+u^3\right).
\end{equation}

We can now state our main result as follows.
\begin{theorem}\label{the1}
Assume that $u_0\in F^{s}_{p,r}$ with $1< p, r<\infty$ and $s>\max\{\frac32,1+\frac1p\}$. Then there exists some time $T>0$ such that
\begin{enumerate}
  \item system \eqref{5}-\eqref{6} has a solution $u\in E^{s}_{p,r}(T):=\mathcal{C}([0,T]; F^s_{p,r})\cap \mathcal{C}^1{([0,T]; F^{s-1}_{p,r})}$;
  \item the solutions of \eqref{5}-\eqref{6} are unique;
  \item the data-to-solution map $u_0 \mapsto u(t)$ is continuous from any bounded
subset of $u_0\in F^s_{p,r}$ into $\mathcal{C}([0,T],F^s_{p,r})$;
\item the data-to-solution map $u_0 \mapsto u(t)$ is not uniformly continuous from any bounded subset of $u_0\in F^s_{p,r}$ into $\mathcal{C}([0,T],F^s_{p,r})$.
\end{enumerate}
\end{theorem}
\begin{remark}\label{re0}
Although the existence and uniqueness of the CH equation has been studied in \cite{Yang}, continuous and non-uniform
continuous dependence with respect to initial data seems empty. Compared with the CH equation, the structure of the Novikov equation is more difficult. Based on this point, the proof of Theorem \ref{the1} is not obvious. Of course, Theorem \ref{the1} holds for the CH equation.
\end{remark}
\begin{remark}\label{re1}
We should mention that,  in \cite{Yang}, the estimate of $\mathcal{B}(f,g)=\pa_x\Lambda^{-2}\left(fg+\fr12\pa_xf\pa_xg\right)$ in $F_{p,r}^{s-1}$ play a key role in the  convergence of the approximating sequence and uniqueness of solutions. However, the bilinear estimates in $F_{p,r}^{s-1}$ is not available due to the lack of product law in $F_{p,r}^{s-2}$. The same question also appears for the  Novikov equation. To overcome this difficulty, our new idea is to work in the weaker space $B_{p,\infty}^{s-1}$ instead of $F_{p,r}^{s-1}$. In fact, when proving the  convergence of the approximating sequence, uniqueness, continuous and non-uniform continuous dependence with respect to initial data, we always encounter the estimation of the  bilinear term $\mathcal{B}(f,g)$ in $B_{p,\infty}^{s-1}$.
\end{remark}
\quad{\bf Notation}\;
 $C$ stands for some positive constant independent of $n$, which may vary from line to line.
The symbol $A\approx B$ means that $C^{-1}B\leq A\leq CB$.
 Given a Banach space $X$, we denote its norm by $\|\cdot\|_{X}$. We shall use the simplified notation $\|f,\cdots,g\|_X=\|f\|_X+\cdots+\|g\|_X$ if there is no ambiguity.
 For $I\subset\R$, we denote by $\mathcal{C}(I;X)$ the set of continuous functions on $I$ with values in $X$.
\section{Preliminaries}\label{sec2}

\qquad For all $f\in \mathcal{S}'$, the Fourier transform $\widehat{f}$ is defined by
$
(\mathcal{F} f)(\xi)=\int_{\R}e^{-\mathrm{i}x\xi}f(x)\dd x$ for any $\xi\in\R.
$
 The inverse Fourier transform of any $g$ is given by
$
(\mathcal{F}^{-1} g)(x)=\check{g}(x)=\frac{1}{2 \pi} \int_{\R} e^{\mathrm{i} x \xi}g(\xi)  \dd \xi.
$
Next, we will recall some facts about the Littlewood-Paley decomposition and the nonhomogeneous Besov spaces (see \cite{BCD} for more details).
Choose a radial, non-negative, smooth function $\vartheta:\R\mapsto [0,1]$ such that
\begin{itemize}
  \item ${\rm{supp}} \;\vartheta\subset B(0, 4/3)$;
  \item $\vartheta(\xi)\equiv1$ for $|\xi|\leq3/4$.
\end{itemize}
Setting $\varphi(\xi):=\vartheta(\xi/2)-\vartheta(\xi)$, then we deduce that $\varphi$ has the following properties
\begin{itemize}
  \item ${\rm{supp}} \;\varphi\subset \left\{\xi\in \R: 3/4\leq|\xi|\leq8/3\right\}$;
  \item $\varphi(\xi)\equiv 1$ for $4/3\leq |\xi|\leq 3/2$;
  \item $\vartheta(\xi)+\sum_{j\geq0}\varphi(2^{-j}\xi)=1$ for any $\xi\in \R$.
\end{itemize}

For every $u\in \mathcal{S'}(\mathbb{R})$, the inhomogeneous dyadic blocks ${\Delta}_j$ are defined as follows
\bbal
\Delta_ju=0,\; \text{if}\; j\leq-2;\quad
\Delta_{-1}u=\vartheta(D)u;\quad
\Delta_ju=\varphi(2^{-j}D)u,\; \; \text{if}\;j\geq0,
\end{align*}
where the pseudo-differential operator $\sigma(D):u\to\mathcal{F}^{-1}(\sigma \mathcal{F}u)$.

The inhomogeneous low-frequency cut-off operator $S_{j}$ is defined by
$$
S_j u=\sum_{-1\leq q\leq j-1}{\Delta}_qu.
$$
Let $s\in\mathbb{R}$ and  $1\leq p<\infty, 1 \leq q \leq \infty$. The inhomogeneous Triebel-Lizorkin spaces $F_{p, q}^s=F_{p, q}^s(\mathbb{R})$ is defined by
$$
F_{p, q}^s:=\left\{f \in \mathcal{S}^{\prime}\left(\mathbb{R}\right),\|f\|_{F_{p, q}^s}<\infty\right\},
$$
where
$$
\|f\|_{F_{p, q}^s}:=\left\|\left(\left|S_0 f\right|^q+\sum_{j \geq 0} 2^{j s q}\left|\Delta_j f\right|^q\right)^{\frac{1}{q}}\right\|_{L_x^p},
$$
Let $\mathcal{S}^{\prime} \backslash \mathcal{P}$ denote the tempered distribution modulo the polynomials, then
$$
\dot{F}_{p, q}^s:=\left\{f \in \mathcal{S}^{\prime} \backslash \mathcal{P},\|f\|_{\dot{F}_{p, q}^s}<\infty\right\},
$$
where
$$
\|f\|_{\dot{F}_{p, q}^s}:=\left\|\left(\sum_{j \in \mathbb{Z}} 2^{j s q}\left|\Delta_j f\right|^q\right)^{\frac{1}{q}}\right\|_{L_x^p} .
$$
We remark that for any $s>0$,
$$
\|f\|_{F_{p, q}^s} \sim\|f\|_{L^p}+\|f\|_{\dot{F}_{p, q}^s}, \quad 1 \leq p<\infty, \quad 1 \leq q \leq \infty .
$$
Let $s\in\mathbb{R}$ and $(p,r)\in[1, \infty]^2$. The nonhomogeneous Besov space $B^{s}_{p,r}(\R)$ is defined by
\begin{align*}
B^{s}_{p,r}(\R):=\Big\{f\in \mathcal{S}'(\R):\;\|f\|_{B^{s}_{p,r}(\mathbb{R})}<\infty\Big\},
\end{align*}
where
\begin{numcases}{\|f\|_{B^{s}_{p,r}(\mathbb{R})}=}
\left(\sum_{j\geq-1}2^{sjr}\|\Delta_jf\|^r_{L^p(\mathbb{R})}\right)^{1/r}, &if $1\leq r<\infty$,\nonumber\\
\sup_{j\geq-1}2^{sj}\|\Delta_jf\|_{L^p(\mathbb{R})}, &if $r=\infty$.\nonumber
\end{numcases}

We shall use Bony's decomposition \cite{BCD} in the nonhomogeneous context throughout this paper
\begin{align*}
uv={T}_{u}v+{T}_{v}u+{R}(u,v)\quad\text{with}
\end{align*}
\begin{eqnarray*}
{T}_{u}v=\sum_{j\geq-1}{S}_{j-1}u{\Delta}_jv \quad\mbox{and} \quad{R}(u,v)=\sum_{|j-k|\leq1}{\Delta}_ju{\Delta}_kv.
\end{eqnarray*}
Next we establish the Moser type inequality for the Triebel-Lizorkin spaces.
\begin{lemma}[\cite{C1}]\label{mos} Let $(p, q) \in(1, \infty) \times(1, \infty]$ or $p=q=\infty, s>0$. There exists some positive constant $C$ with the following property:
$$
\|f g\|_{\dot{F}_{p, q}^s} \leq C\left(\|f\|_{L^{\infty}}\|g\|_{\dot{F}_{p, q}^s}+\|g\|_{L^{\infty} \|}\|f\|_{\dot{F}_{p, q}^s}\right)
$$
\end{lemma}
\begin{lemma}[\cite{BCD}]\label{le1}
Let $(p,r)\in[1, \infty]^2$ and $s>\max\big\{1+\frac1p,\frac32\big\}$. Then we have
\bbal
&\|fg\|_{B^{s-2}_{p,r}}\leq C\|g\|_{B^{s-2}_{p,r}}\|g\|_{B^{s-1}_{p,r}}.
\end{align*}
\end{lemma}
\begin{lemma}\label{com1}
  Let $(p,q)\in (1,\infty)\times (1,\infty]$ or $p=q=\infty$. Then for $s>0$
\bal\label{appendix-A}
  \big\| \|2^{ks}([f,\Delta_k]\cdot \nabla g)\|_{\ell^q(\mathbb{Z})} \big\|_{L^p} \leq C\big(\|\nabla f\|_\infty \|g\|_{\F^s_{p,q}}+\|\nabla g\|_\infty \|f\|_{\F^s_{p,q}} \big)
\end{align}
and
\bal\label{appendix-B}
  \big\| \|2^{ks}([f,\Delta_k]\cdot \nabla g)\|_{\ell^q(\mathbb{Z})} \big\|_{L^p} \leq C\big(\|\nabla f\|_\infty \|g\|_{\F^s_{p,q}}+\|g\|_{\infty} \|\nabla f\|_{\F^s_{p,q}} \big).
\end{align}
\end{lemma}

\begin{proof}  We use the Einstein convention on the summation over repeated indices $i \in [1,d]$. By the Bony's decomposition, one has
\begin{align*}
[f,\Delta_k]\cdot \nabla g &=[f_i,\Delta_k]\partial_i g=[T_{f_i},\Delta_k]\partial_i g + T'_{\Delta_k \partial_i g}f_i-\Delta_k(T_{\partial_i g}f_i)-\Delta_k(R(f_i,\partial_i g)) \\
&\triangleq \uppercase\expandafter{\romannumeral1}+\uppercase\expandafter{\romannumeral2}+\uppercase\expandafter{\romannumeral3}+\uppercase\expandafter{\romannumeral4},
\end{align*}
where $T'_u v$ stands for $T_u v+R(u,v)$. The proof is similar with \cite{CMZ}. We just estimate the terms $\uppercase\expandafter{\romannumeral1}$ and $\uppercase\expandafter{\romannumeral4}$ since the condition $\div f=0$ is used in \cite{CMZ}.
\begin{align}\label{A-1}
|\uppercase\expandafter{\romannumeral1}|&=\left| \sum_{k'\sim k}\int_{\R^d}(S_{k'-1}f_i(x)-S_{k'-1}f_i(y))2^{kd}h(2^k(x-y))\partial_i\Delta_{k'}g(y) \dd y\right|\nonumber \\
&\lesssim  \sum_{k'\sim k}\int_{\R^d}\big|(S_{k'-1}\partial_i f_i(x)-S_{k'-1}\partial_i f_i(y))2^{kd}h(2^k(x-y))\Delta_{k'}g(y)\big|\dd y \nonumber \\
&\quad +  \sum_{k'\sim k}\int_{\R^d} \big|(S_{k'-1} f_i(x)-S_{k'-1} f_i(y))2^{k(d+1)}(\partial_i h)(2^k(x-y))\Delta_{k'}g(y)\big| \dd y \nonumber \\
&\lesssim \|\nabla f\|_\infty \int_{\R^d} 2^{kd}|\nabla h(2^k(x-y))| |\Delta_{k'}g(y)|\dd y \nonumber \\
&\quad+ \|\nabla f\|_\infty \int_{\R^d}2^k |x-y|2^{kd}|\nabla h(2^k(x-y))| |\Delta_{k'}g(y)|\dd y \nonumber \\
&\lesssim \|\nabla f\|_\infty \mathbf{M}(|\Delta_{k'}g(\cdot)|)(x),
\end{align}
where $k'\sim k$ stands for $|k'-k|\leq 4$ and the maximal function $Mf(x)$ is defined by
$$\mathbf{M}f(x)=\sup_{r>0}\frac{1}{|\mathfrak{B}(x,r)|}\int_{\mathfrak{B}(x,r)}f(y)\dd y.$$
Multiplying $2^{ks}$ on both sides of \eqref{A-1}, taking $\ell^q(\mathbb{Z})$ norm then taking $L^p$ norm, we have
\begin{align*}
\big\| \|2^{ks}|\uppercase\expandafter{\romannumeral1}(x)|\|_{\ell^q(\mathbb{Z})} \big\|_{L^p} &\lesssim \|\nabla f\|_\infty \left\| \|\mathbf{M}(2^{k's}|\Delta_{k'}g(\cdot)|)(x)\|_{\ell^q(\mathbb{Z})} \right\|_{L^p} \nonumber\\
&\lesssim \|\nabla f\|_\infty \big\| \|2^{ks}|\Delta_{k}g(x)|\|_{\ell^q(\mathbb{Z})} \big\|_{L^p} \nonumber\\
&\lesssim \|\nabla f\|_\infty \|g\|_{\F^s_{p,q}}.
\end{align*}
For the term $\uppercase\expandafter{\romannumeral4}$, we have
\begin{align*}
|\uppercase\expandafter{\romannumeral4}|&=\left| \sum_{k'\geq k-3} \Delta_k(\Delta_{k'}f_i\partial_i\tilde{\Delta}_{k'}g)\right|\nonumber \\
&=\left| \sum_{k'\geq k-3}\int_{\R^d} 2^{kd}h(2^k(x-y)) \Delta_{k'}f_i(y)\partial_i\tilde{\Delta}_{k'}g(y)\dd y \right|\nonumber \\
&\lesssim  \sum_{k'\geq k-3}\int_{\R^d}\big|2^{kd}h(2^k(x-y)) \Delta_{k'}\partial_i f_i(y)\tilde{\Delta}_{k'}g(y)\big|\dd y \nonumber \\
&\quad +  \sum_{k'\geq k-3}\int_{\R^d} \big|2^{k(d+1)}(\partial_i h)(2^k(x-y)) \Delta_{k'}f_i(y)\partial_i\tilde{\Delta}_{k'}g(y)\big| \dd y  \nonumber \\
&\lesssim \|\nabla f\|_\infty \sum_{k'\geq k-3} \mathbf{M}(\tilde{\Delta}_{k'}g)(x) + \sum_{k'\geq k-3} 2^k \mathbf{M}(\tilde{\Delta}_{k'}g)(x) \|\Delta_{k'}f \|_\infty.
\end{align*}
Then for $s>0$, we obtain
\begin{align*}
\big\| \|2^{ks}|\uppercase\expandafter{\romannumeral4}(x)|\|_{\ell^q(\mathbb{Z})} \big\|_p &\lesssim \|\nabla f\|_\infty \left\| \left\| \sum_{k'\geq k-3} 2^{(k-k')s}\mathbf{M}(2^{k's}\tilde{\Delta}_{k'}g)(x) \right\|_{\ell^q(\mathbb{Z})} \right\|_{L^p} \nonumber\\
&\quad+ \|\nabla \Delta_{k'} f\|_\infty \big\| \| \sum_{k'\geq k-3} 2^{(k-k')(s+1)}\mathbf{M}(2^{k's}\tilde{\Delta}_{k'}g)(x) \|_{\ell^q(\mathbb{Z})} \big\|_{L^p} \nonumber\\
&\lesssim \|\nabla f\|_\infty \big\| \|\mathbf{M}(2^{k's}\tilde{\Delta}_{k'}g)(x)\|_{\ell^q(\mathbb{Z})} \big\|_{L^p} \nonumber\\
&\lesssim \|\nabla f\|_\infty \big\|\| 2^{ks}|\tilde{\Delta}_{k}g(x)| \|_{\ell^q(\mathbb{Z})}\big\|_{L^p} \nonumber\\
&\lesssim \|\nabla f\|_\infty \|g\|_{\F^s_{p,q}}.
\end{align*}
Therefore, we get the desired inequality \eqref{appendix-A}. The inequality \eqref{appendix-B} is similar, we omit here.
\end{proof}
\begin{proposition}\label{2-1}
Let $1< p,q<\infty$ and $s>0$.  Assume that $f_0\in F^s_{p,r}(\mathbb{R}^d)$, $g\in L^1([0,T]; F^s_{p,r}(\mathbb{R}^d))$. If $f\in L^\infty([0,T]; F^s_{p,r}(\mathbb{R}^d))\bigcap \mathcal{C}([0,T]; \mathcal{S}'(\mathbb{R}^d))$ solves the following d-D linear transport equation:
\begin{equation}\label{2.14}
\quad \partial_t f+v\cdot \nabla f=g,\quad  f|_{t=0} =f_0,
\end{equation}
then there exists a constant $C=C(d,p,r,s)$, such that
\begin{align}\label{trans1}
&\|f(t)\|_{\F^s_{p,r}(\mathbb{R}^d)}\leq e^{CV(t)}\|f_0\|_{\F^s_{p,r}(\mathbb{R}^d)}+\int_0^t\|g(\tau)\|_{\F^s_{p,r}(\mathbb{R}^d)}\mathrm{d}\tau \nonumber\\& \quad \quad
+C\int^t_0\left(\|f(\tau)\|_{\F^s_{p,r}(\mathbb{R}^d)}\|\nabla v(\tau)\|_{L^\infty(\mathbb{R}^d)}+\|\nabla v(\tau)\|_{\F^{s-1}_{p,r}(\mathbb{R}^d)}\|\nabla f(\tau)| |_{L^\infty(\mathbb{R}^d)}\right)\mathrm{d}\tau,
\end{align}
or
\begin{align}\label{trans2}
&\|f(t)\|_{\F^s_{p,r}(\mathbb{R}^d)}\leq e^{CV(t)}\|f_0\|_{\F^s_{p,r}(\mathbb{R}^d)}+\int_0^t\|g(\tau)\|_{\F^s_{p,r}(\mathbb{R}^d)}\mathrm{d}\tau \nonumber\\& \quad \quad
+C\int^t_0\left(\|f(\tau)\|_{\F^s_{p,r}(\mathbb{R}^d)}\|\nabla v(\tau)\|_{L^\infty(\mathbb{R}^d)}+\|\nabla v(\tau)\|_{\F^{s}_{p,r}(\mathbb{R}^d)}\|f(\tau)| |_{L^\infty(\mathbb{R}^d)}\right)\mathrm{d}\tau,
\end{align}
here $V(t)=\int^t_0\|\pa_xv\|_{L^\infty}\dd \tau$.

Moreover, if $s>1+\frac{d}{p}$, we have
\bbal
\|f(t)\|_{{F}^s_{p,r}}\leq \left(\|f_0\|_{{F}^s_{p,r}}+\int^t_0e^{-C\widetilde{V}(\tau)}\|g(\tau)\|_{F^s_{p,r}}\dd \tau\right)e^{C\widetilde{V}(t)},
\end{align*}
here $\widetilde{V}(t)=\int^t_0\|v\|_{F^s_{p,r}}\dd \tau$.
\end{proposition}

\begin{proof}  Let $X(t,x)$ solves the ordinary differential equation
$$\pa_tX(t,x)=v(t,X(t,x)),\quad X(0,x)=x.$$
Then
\bbal
\|\pa_xX(t,x)\|_{L^\infty}+\|(\pa_xX)^{-1}(t,x)\|_{L^\infty}\leq e^{V(t)}, \quad V(t)=\int^t_0\|\pa_xv\|_{L^\infty}\dd \tau.
\end{align*}
Applying $\Delta_j$ to \eqref{2.14}, we have
\bbal
\pa_t\Delta_jf+v\pa_x\Delta_jf=\Delta_jg+[v,\Delta_j]\pa_xf,
\end{align*}
which implies
\bal\label{trans}
\pa_t(\Delta_jf(t,X(t,x)))=\Delta_jg(t,X(t,x))+[v,\Delta_j]\pa_xf(t,X(t,x)).
\end{align}
Then, we obtain
\bbal
\left(\sum_{j}|2^{js}\Delta_jf(t,X(t,x))|^r\right)^{\frac1r}&\leq \left(\sum_{j}|2^{js}\Delta_jf_0(x)|^r\right)^{\frac1r}
+\int^t_0 \left(\sum_{j}|2^{js}\Delta_jg(\tau,X(t,x))|^r\right)^{\frac1r}\dd \tau \\&\quad+\int^t_0\left(\sum_{j}|2^{js}[v,\Delta_j]\pa_xf(\tau,X(t,x))|^r\right)^{\frac1r}\dd \tau,
\end{align*}
by Lemma\ref{com1}, we get \eqref{trans1} and \eqref{trans2}. Moreover, by \eqref{trans},
\bbal
\pa_t(e^{-V(t)}\Delta_jf(t,X(t,x)))&=-e^{-V(t)}\|\pa_xv(t)\|_{L^\infty}\Delta_jf(t,X(t,x)) \\ &\quad +e^{-V(t)}\Delta_jg(t,X(t,x))+e^{-V(t)}[v,\Delta_j]\pa_xf(t,X(t,x)).
\end{align*}
Then, we obtain
\bbal
&e^{-V(t)} \left(\sum_{j}|2^{js}\Delta_jf(t,X(t,x))|^r\right)^{\frac1r}\\
&\leq \left(\sum_{j}|2^{js}\Delta_jf_0(x)|^r\right)^{\frac1r}+\int^t_0 e^{-V(\tau)}\|\pa_xv(\tau)\|_{L^\infty}\left(\sum_{j}|2^{js}\Delta_jf(\tau,X(t,x))|^r\right)^{\frac1r}\dd \tau
\\&+\int^t_0 e^{-V(\tau)}\left(\sum_{j}|2^{js}\Delta_jg(\tau,X(t,x))|^r\right)^{\frac1r}\dd \tau +\int^t_0e^{-V(\tau)}\left(\sum_{j}|2^{js}[v,\Delta_j]\pa_xf(\tau,X(t,x))|^r\right)^{\frac1r}\dd \tau,
\end{align*}
which implies
\bbal
\|f(t)\|_{\dot{F}^s_{p,r}}&\leq e^{V(t)}\left(e^{CV(t)}\|f_0\|_{\dot{F}^s_{p,r}}+\int^t_0e^{-V(\tau)}\|g(\tau)\|_{\dot{F}^s_{p,r}}\dd \tau
 +C\int^t_0e^{-V(\tau)}\|v\|_{F^s_{p,r}}\|f\|_{F^s_{p,r}}\dd \tau\right).
\end{align*}
Also, we have
\bbal
\|f(t)\|_{L^p}\leq e^{V(t)}\left(e^{CV(t)}\|f_0\|_{L^p}+\int^t_0e^{-V(\tau)}\|\pa_xv(\tau)\|_{L^\infty}\|f(\tau)\|_{L^p}\dd \tau+\int^t_0e^{-V(\tau)}\|g(\tau)\|_{L^p}\dd \tau\right).
\end{align*}
Thus
\bbal
e^{-V(t)}\|f(t)\|_{{F}^s_{p,r}}&\leq e^{CV(t)}\|f_0\|_{{F}^s_{p,r}}+\int^t_0e^{-V(\tau)}\|g(\tau)\|_{F^s_{p,r}}\dd \tau
 +C\int^t_0e^{-V(\tau)}\|v\|_{F^s_{p,r}}\|f\|_{F^s_{p,r}}\dd \tau,
\end{align*}
which implies
\bbal
e^{-V(t)}\|f(t)\|_{{F}^s_{p,r}}\leq\left(e^{CV(t)}\|f_0\|_{{F}^s_{p,r}}++\int^t_0e^{-V(\tau)}\|g(\tau)\|_{F^s_{p,r}}\dd \tau\right)e^{C\int^t_0\|v\|_{F^s_{p,r}}\dd \tau},
\end{align*}
and then leads to
\bbal
\|f(t)\|_{{F}^s_{p,r}}\leq \left(\|f_0\|_{{F}^s_{p,r}}+\int^t_0e^{-C\widetilde{V}(\tau)}\|g(\tau)\|_{F^s_{p,r}}\dd \tau\right)e^{C\widetilde{V}(t)}.
\end{align*}
This completes the proof of Proposition \ref{2-1}.
\end{proof}
\section{Proof of Theorem \ref{the1}}
We divide the proof of Theorem \ref{the1} into several steps.
\subsection{Existence and Uniqueness}

\textbf{Step 1. Construction of the approximate system}

We define by induction a sequence $\{u^{n}\}_{n\in\mathbb{N}}$ of smooth functions
  by solving the following linear system:
    \begin{align}\label{lin}
       \begin{cases}
        \partial_tu^{n+1}+(u^n)^2\partial_xu^{n+1}=\mathbf{P}(u^n),\\
        u^{n+1}(0,x)=S_{n+1}u_0(x),
        \end{cases}
    \end{align}
for $n=0,1,2,\cdots$, where $u^{0}\triangleq0$.

Since $u_0\in F^{s}_{p,r}$, then all initial data $S_{n+1}u_0\in F^\infty_{p,r}$ and $\|S_{n+1}u_0\|_{F^{s}_{p,r}}\leq C\|u_0\|_{ F^{s}_{p,r}}$. By induction, for every $n\geq 1$,  there exist a  maximal time $T_n$ and a unique solution $u^{n}$ to \eqref{lin} in  $\mathcal{C}^1([0,T_n);F^\infty_{p,r})$. Obviously, $u^{n}$  belongs to $E^{s}_{p,r}(T)$ for all positive $T$.

\textbf{Step 2. Uniform estimates to the approximate solutions}

Since $\pa_x\Lambda^{-2}$ is $S^{-1}$-multiplier, we have the following inequality
   \begin{align}\label{l}
   \|u^{n+1}(t)\|_{F^{s}_{p,r}}&\leq e^{CU^{n}(t)}\left(\|S_{n+1}u_0\|_{F^{s}_{p,r}}+C\int^t_0 e^{-CU^{n}(\tau)}
   \|\mathbf{P_1}(u^{n}),\mathbf{P_2}(u^{n})\|_{F^{s}_{p,r}} \mathrm{d}\tau\right)\nonumber\\
   &\leq Ce^{CU^{n}(t)} \|u_0\|_{F^{s}_{p,r}}+C\int^t_0 e^{CU^{n}(t)-CU^{n}(\tau)}
   \|u^{n}(\tau)\|^3_{F^{s}_{p,r}}\mathrm{d}\tau,
   \end{align}
   where  $U^{n}(t)\triangleq \int_0^t \|u^{n}(\tau)\|^2_{F^{s}_{p,r}}\mathrm{d}\tau$ and $C\geq1$.

We now fix a $T>0$ such that $4C^3\|u_0\|_{F^{s}_{p,r}}T<1$. By induction, we gain
\begin{align*}\label{3.2}
\forall t\in[0,T],~~~\|u^n(t)\|_{F^{s}_{p,r}}
\leq \frac{C\|u_0\|_{F^{s}_{p,r}}}{\sqrt{1-4C^3\|u_0\|^2_{F^{s}_{p,r}}t}},~~~~~~\forall n\in \mathbb{N}.
\end{align*}
In fact, this is obvious for $n=0$ since $u^0 = 0$. Suppose it is valid for $n\geq1$, we shall prove it holds for $n+1$.
Direct computations gives that
\begin{align*}
CU^{n}(t)-CU^{n}(t')&=C\int^{t}_{t'}\|u^{n}(\tau)\|^2_{F^{s}_{p,r}}\mathrm{d}\tau\\
&\leq \int^{t}_{t'}\frac{C^3\|u_0\|^2_{F^{s}_{p,r}}}{1-4C^3\|u_0\|^2_{F^{s}_{p,r}}\tau}\mathrm{d}\tau\\&
=\fr14\ln\left(\frac{1-4C^3\|u_0\|^2_{F^{s}_{p,r}}t'}{1-4C^3\|u_0\|^2_{F^{s}_{p,r}}t}\right),
\end{align*}
then from \eqref{l}, we deduce that
\begin{align*}
\|u^{n+1}(t)\|_{F^{s}_{p,r}}\nonumber &\leq C\left(1-4C^3\|u_0\|^2_{F^{s}_{p,r}}t\right)^{-\fr14}\|u_0\|_{_{F^s_{p,r}}}\left(1
+\int^t_0\left(1-4C^3\|u_0\|^2_{F^{s}_{p,r}}t'\right)^{-\fr54}C^3\|u_0\|^{2}_{F^s_{p,r}}\dd t'\right)
\\\nonumber&= \frac{C\|u_0\|_{F^{s}_{p,r}}}{\sqrt{1-4C^3\|u_0\|^2_{F^{s}_{p,r}}t}}\leq \frac{C\|u_0\|_{F^{s}_{p,r}}}{\sqrt{1-4C^3\|u_0\|^2_{F^{s}_{p,r}}T}}.
\end{align*}
Therefore, $\{u^{n}\}_{n\in\mathbb{N}}$ is bounded in $L^\infty([0,T];F^{s}_{p,r})$. This entails that $(u^{n})^2\partial_xu^{n+1}$ is bounded  in  $L^\infty([0,T];F^{s-1}_{p,r}).$  As the right-hand side of \eqref{l} is bounded in $L^\infty([0,T];F^{s}_{p,r}),$  we can conclude that $\{u^{n}\}_{n\in\mathbb{N}}$ is bounded in $E^{s}_{p,r}(T).$

\textbf{Step 3. Convergence of the approximate solutions}

   We will show that $\{u^{n}\}_{n\in\mathbb{N}}$ is a Cauchy sequence in $\mathcal{C}([0,T];B^{s-1}_{p,\infty})$. Indeed, for all $n\in\mathbb{N},$  we have
 \begin{align*}
    \begin{cases}
    \left(\partial_t+(u^{n+m})^2\partial_x\right)(u^{n+m+1}-u^{n+1})=-(u^{m+n}-u^n)(u^{m+n}+u^n)\partial_xu^{n+1} +\mathbf{P}(u^{n+m})-\mathbf{P}(u^n),\\
    (u^{n+m+1}-u^{n+1})|_{t=0}=(S_{n+m+1}-S_{n+1})u_0(x).
    \end{cases}
 \end{align*}
By the definition of $S_n$, we have
\begin{equation*}
   (S_{n+m+1}-S_{n+1})u_0(x)=\sum\limits_{q=n+1}^{n+m} \Delta_q u_0(x),
\end{equation*}
which leads to
\begin{align*}
\left\|\sum\limits_{q=n+1}^{n+m} \Delta_q u_0\right\|_{B^{s-1}_{p,\infty}}&=\sup_{j\geq-1}2^{j(s-1)}\left\|\Delta_j\sum\limits_{q=n+1}^{n+m} \Delta_q u_0\right\|_{L^p}\\
&\leq C\sum\limits_{q=n+1}^{n+m} 2^{-q}2^{qs}\left\|\Delta_q u_0\right\|_{L^p}\leq C2^{-n}\|u_0\|_{B^{s}_{p,\infty}}\leq C2^{-n}\|u_0\|_{F^{s}_{p,r}}.
\end{align*}
Applying the regularity theory of transport equation, we obtain for $t\in[0,T]$
\begin{align*}
  \|(u^{n+m+1}-u^{n+1})(t)\|_{B^{s-1}_{p,\infty}}
  & \leq C
  e^{CU^{n+m}(t)}\Big(\|(S_{n+m+1}-S_{n+1})u_0\|_{B^{s-1}_{p,\infty}}\\
  &\quad+\int^t_0e^{-CU^{n+m}(\tau)}\|(u^{n+m}-u^{n})(\tau)\|_{B^{s-1}_{p,\infty}}
  \|(u^n,u^{n+1},u^{n+m},u^{n+m+1})(\tau)\|^2_{B^s_{p,\infty}}\dd \tau\Big).
\end{align*}
Define
\begin{equation*}
   b_n^m(t)\triangleq \|(u^{n+m}-u^n)(t)\|_{B^{s-1}_{p,\infty}}.
 \end{equation*}
So we can find a positive $C_T$ independent of $n,m$ such that
\begin{align*}
\begin{cases}
b^m_{n+1}(t)\leq C_{T}\left(2^{-n}+\int_0^t b^m_n(\tau)\mathrm{d}\tau\right),\quad \forall \, t\in[0,T],\\
b^m_1(t)\leq C_T(1+t),\quad \forall \, t\in[0,T].
\end{cases}
\end{align*}
Arguing by induction with respect to the index $n$, we deduce that
\begin{align*}
b^m_{n+1}(t)
&\leq\left(C_{T}\sum\limits_{k=0}^n \frac{(2T C_T)^k}{k!}\right)2^{-n}+ C_T\frac{(T C_T)^{n+1}}{(n+1)!}\rightarrow 0 ,\quad as\ n\rightarrow \infty.
\end{align*}
So we conclude that $\{u^{n}\}_{n\in\mathbb{N}}$ is a Cauchy sequence in $\mathcal{C}([0,T];B^{s-1}_{p,\infty})$ and converges to a limit function $u\in \mathcal{C}([0,T];B^{s-1}_{p,\infty})$.

\textbf{Step 4. Local existence of a solution}

We will show that $u$ belongs to $E^{s}_{p,r}(T)$ and satisfies Eqs.\eqref{5}-\eqref{6}. Since  $\{u^{n}\}_{n\in \mathbb{N}}$ is bounded in $L^\infty(0,T;F^{s}_{p,r})$, Fatou's lemma guarantees that $u$ also belongs to $L^\infty(0,T;F^{s}_{p,r})$. Now, as $\{u^{n}\}_{n\in\mathbb{N}}$
converges to $u \in \ \mathcal{C}([0,T];B^{s-1}_{p,\infty})\hookrightarrow\ \mathcal{C}([0,T];F^{s-1-\eta}_{p,r})$ with $\eta\in(0,s-1)$, an interpolation argument ensures that the convergence actually holds true in $  \mathcal{C}([0,T];F^{s'}_{p,r})$ for any $s' < s.$  It is then easy to pass to the limit in \eqref{lin} and to conclude that
$u$ is a solution to \eqref{lin}.
Furthermore, the solution $u$ is in $\mathcal{C}([0,T];F^{s}_{p,r})$ if $r<\infty$. In fact, for any $t_1,t_2\in [0,T]$, we have
\bbal
\|\mathbf{S}_{t_2}(u_0)-\mathbf{S}_{t_1}(u_0)\|_{F^s_{p,r}}
&\leq \|\mathbf{S}_{t_1}(S_Nu_0)-\mathbf{S}_{t_1}(u_0)\|_{F^s_{p,r}}+\|\mathbf{S}_{t_1}(S_Nu_0)-\mathbf{S}_{t_2}(S_Nu_0)\|_{F^s_{p,r}}\\
&\quad+\|\mathbf{S}_{t_2}(S_Nu_0)-\mathbf{S}_{t_2}(u_0)\|_{F^s_{p,r}}\\
&\leq \sum_{i=1}^2\|\mathbf{S}_{t_i}(S_Nu_0)-\mathbf{S}_{t_i}(u_0)\|_{F^s_{p,r}}+\int_{t_1}^{t_2}\|\pa_{\tau}\mathbf{S}_{\tau}(S_Nu_0)\|_{F^s_{p,r}}\dd\tau\\
&\leq C\|S_Nu_0-u_0\|_{F^s_{p,r}}+C2^N|t_1-t_2|.
\end{align*}
Combining  the density of $S_Nu_0$ in $F^s_{p,r}$ with $r<\infty$  yields that $u\in\mathcal{C}([0,T];F^{s}_{p,r})$.

Finally, because $u$ belongs to $ L^\infty(0,T;F^{s}_{p,r})$, then the right-hand side of the following equation
$$ \pa_tu +u\pa_xu=\mathbf{P}(u)\in L^\infty(0,T;F^{s}_{p,r}).$$
Also, we see that $\partial_tu$
is in $\mathcal{C}([0,T];F^{s-1}_{p,r})$ if $r<\infty$. Hence we conclude that the solution $ u\in \mathcal{C}([0,T];F^{s}_{p,r})\cap \mathcal{C}^1([0,T];F^{s-1}_{p,r}).$

\textbf{Step 5. Uniqueness of a solution}

In order to prove the uniqueness, we consider two solutions $u$ and $v$ of Eqs.\eqref{5}-\eqref{6} in
$E^{s}_{p,r}(T)$ with the same initial data. Using the following Lemma, we must have $u\equiv v$ in $E^{s}_{p,r}(T)$. In fact,
$$\|u-v\|_{B^{s-1}_{p,\infty}}\leq C\|u_1(0)-u_2(0)\|_{B^{s-1}_{p,\infty}}=0,$$
which implies the uniqueness.

\begin{lemma}\label{ley1} Let $u,v\in \mathcal{C}([0,T],F^{s}_{p,r})$ be two solutions of Eqs.\eqref{5}-\eqref{6} associated with $u_0$ and $v_0$, respectively. Then we have the estimate for the difference $w=u-v$
\begin{align}
&\|w\|_{B^{s-1}_{p,\infty}}\leq C\|w_0\|_{B^{s-1}_{p,\infty}},\label{y1}
\\&\|w\|_{F^{s}_{p,r}}\leq C\left(\|w_0\|_{F^{s}_{p,r}}+\int^t_0\|\pa_{x}v\|_{F^{s}_{p,r}}\|w\|_{L^\infty}\dd \tau\right),\label{y2}
\end{align}
where $w_0=u_0-v_0$, the constants $C$ depends on $T$ and initial norm $\|u_0,v_0\|_{F^{s}_{p,r}}$.
\end{lemma}
\begin{proof}\;
Obviously, $w\in \mathcal{C}([0,T],F^{s}_{p,r})$ and $w$ solves the transport equation
\begin{align}\label{m}
\begin{cases}
\pa_tw+u^2w_x=-w(u+v)v_x+[\mathbf{P}(u)-\mathbf{P}(v)], \\
w(0,x)=u_0(x)-v_0(x).
\end{cases}
\end{align}
Applying Proposition \ref{2-1} to \eqref{m} and using Lemma \ref{mos} yields \eqref{y2}. It is easy to obtain \eqref{y1} and we omit the proof.
\end{proof}
\subsection{Continuous Dependence}
Letting $u=\mathbf{S}_{t}(u_0)$ and $v=\mathbf{S}_{t}(S_Nu_0)$, using Lemma \ref{ley1}, we have
\bbal
&\|\mathbf{S}_{t}(S_Nu_0)\|_{F^{s+1}_{p,r}}\leq C\|S_Nu_0\|_{F^{s+1}_{p,r}}\leq C2^N\|u_0\|_{F^{s}_{p,r}}
\end{align*}
and
\bbal
\|\mathbf{S}_{t}(S_Nu_0)-\mathbf{S}_{t}(u_0)\|_{L^\infty}&\leq C\|\mathbf{S}_{t}(S_Nu_0)-\mathbf{S}_{t}(u_0)\|_{B^{s-1}_{p,\infty}}
\\&\leq C\|S_Nu_0-u_0\|_{B^{s-1}_{p,\infty}}\\
&\leq C\|S_Nu_0-u_0\|_{F^{s-1}_{p,r}}\\
&\leq C2^{-N}\|S_{N}u_0-u_0\|_{F^{s}_{p,r}},
\end{align*}
which implies
\bal\label{zm}
\|\mathbf{S}_{t}(S_Nu_0)-\mathbf{S}_{t}(u_0)\|_{F^s_{p,r}}&\leq C\|S_Nu_0-u_0\|_{F^s_{p,r}}.
\end{align}
Then, using the triangle inequality, we have for $u_0,\widetilde{u}_0\in F^s_{p,r}$,
\bbal
\|\mathbf{S}_{t}(u_0)-\mathbf{S}_{t}(\widetilde{u}_0)\|_{F^s_{p,r}}
&\leq \|\mathbf{S}_{t}(S_Nu_0)-\mathbf{S}_{t}(u_0)\|_{F^s_{p,r}}+\|\mathbf{S}_{t}(S_N\widetilde{u}_0)-\mathbf{S}_{t}(\widetilde{u}_0)\|_{F^s_{p,r}}\\
&~~+\|\mathbf{S}_{t}(S_Nu_0)-\mathbf{S}_{t}(S_N\widetilde{u}_0)\|_{F^s_{p,r}}
\\&\leq C\|S_Nu_0-u_0\|_{F^s_{p,r}}+C\|S_N\widetilde{u}_0-\widetilde{u}_0\|_{F^s_{p,r}}+C\|\mathbf{S}_{t}(S_Nu_0)-\mathbf{S}_{t}(S_N\widetilde{u}_0)\|_{F^s_{p,r}}
\\
&:=\mathbf{I}_1+\mathbf{I}_2+\mathbf{I}_3.
\end{align*}
By the interpolation inequality, one has
\bbal
\mathbf{I}_3&\leq C\|\mathbf{S}_{t}(S_Nu_0)-\mathbf{S}_{t}(S_N\widetilde{u}_0)\|_{B^s_{p,1}}\\&\leq C\|\mathbf{S}_{t}(S_Nu_0)-\mathbf{S}_{t}(S_N\widetilde{u}_0)\|^{\fr12}_{B^{s-1}_{p,\infty}}
\|\mathbf{S}_{t}(S_Nu_0)-\mathbf{S}_{t}(S_N\widetilde{u}_0)\|^{\fr12}_{B^{s+1}_{p,\infty}}\\
&\leq C\|S_Nu_0-S_N\widetilde{u}_0\|^{\fr12}_{B^{s-1}_{p,\infty}}
\|\mathbf{S}_{t}(S_Nu_0)-\mathbf{S}_{t}(S_N\widetilde{u}_0)\|^{\fr12}_{B^{s+1}_{p,\infty}}\leq C2^{\frac N2}\|u_0-\widetilde{u}_0\|^{\frac12}_{F^{s-1}_{p,r}},
\end{align*}
which clearly implies
\bbal
\|\mathbf{S}_{t}(u_0)-\mathbf{S}_{t}
(\widetilde{u}_0)\|_{F^s_{p,r}}
&\lesssim \|S_Nu_0-u_0\|_{F^s_{p,r}}+
\|S_N\widetilde{u}_0-\widetilde{u}_0\|_{F^s_{p,r}}+2^{\frac N2}\|u_0-\widetilde{u}_0\|^{\frac12}_{F^{s-1}_{p,r}}.
\end{align*}
Due to $1<p,r<\infty$, $\forall \varepsilon>0$, one can select $N$ to be sufficiently large, such that
$$C(\|S_Nu_0-u_0\|_{F^s_{p,r}}+
\|S_N\widetilde{u}_0-\widetilde{u}_0\|_{F^s_{p,r}})\leq \frac{\varepsilon}{2}.$$
Then fix $N$, choose $\delta$ so small that $\|u_0-\widetilde{u}_0\|^{\frac12}_{F^{s-1}_{p,r}}<\delta$ and $C2^{\frac N2}\delta^{\fr12}<\frac{\varepsilon}{2}$. Hence
\bbal
\|\mathbf{S}_{t}(u_0)-\mathbf{S}_{t}
(\widetilde{u}_0)\|_{F^s_{p,r}}<\varepsilon.
\end{align*}
This completes the proof of continuous dependence.

\subsection{Non-uniform Continuous Dependence}
{\bf Step1: Error estimates between the data-to-solution map and initial data}
\begin{proposition}\label{pro 3.1}
Assume that $s>\max\{1+\frac{1}{p}, \frac{3}{2}\}$ with $1< p, r< \infty$ and $\|u_0\|_{F^s_{p,r}}\lesssim 1$. Then under the assumptions of Theorem \ref{the1}, we have
\bbal
\|\mathbf{S}_{t}(u_0)-u_0+t\mathbf{v}_0\|_{F^{s}_{p,r}}\lesssim t^{2}\left[1+\|u_0\|^2_{F^{s-1}_{p,r}}\|u_0\|_{F^{s+1}_{p,r}}
+\|u_0\|^3_{F^{s-1}_{p,r}}\Big(\|u_0\|_{F^{s+1}_{p,r}}+\|u_0\|_{F^{s-1}_{p,r}}\|u_0\|_{F^{s+2}_{p,r}}\Big)\right],
\end{align*}
where we denote $\mathbf{v}_0=u_0^2\pa_xu_0-\mathbf{P_1}(u_0)-\mathbf{P_2}(u_0)$.
\end{proposition}
\begin{proof}\; For simplicity, we denote $u(t)=\mathbf{S}_t(u_0)$. Firstly, by the local well-posedness result, there exists a positive time $T=T(\|u_0\|_{F^s_{p,r}},s,p,r)$ such that the solution $u(t)$ belongs to $\mathcal{C}([0, T];  B_{p, r}^s)$. Moreover,  for all $t\in[0,T]$ and $\gamma\geq s$,  there holds
\bal\label{u-estimate}
\|u(t)\|_{B^{s-1}_{p,\infty}}\leq C\|u_0\|_{B^{s-1}_{p,\infty}},\quad \|u(t)\|_{F^\gamma_{p,r}}\leq C\|u_0\|_{F^\gamma_{p,r}}.
\end{align}

Now we shall estimate the different Triebel-Lizorkin and Besov norms of the term $u(t)-u_0$.
By the Mean Value Theorem, we obtain
\bal\label{zyl1}
\|u(t)-u_0\|_{F^s_{p,r}}
&\leq \int^t_0\|\pa_\tau u\|_{F^s_{p,r}} \dd\tau\nonumber\\
&\leq \int^t_0\|\mathbf{P_1}(u)\|_{F^s_{p,r}} \dd\tau+\int^t_0\|\mathbf{P_2}(u)\|_{F^s_{p,r}} \dd\tau+ \int^t_0\|u^2 \pa_xu\|_{F^s_{p,r}} \dd\tau\nonumber\\
&\lesssim t\left(\|u\|^3_{L_t^\infty(F^{s}_{p,r})}
+\|u\|^2_{L_t^\infty(L^\infty)}\|u_x\|_{L_t^\infty(F^{s}_{p,r})}+\|u_x\|_{L_t^\infty(L^\infty)}\|u\|^2_{L_t^\infty(F^{s}_{p,r})}\right)\nonumber\\
&\lesssim t\left(\|u_0\|^3_{F^{s}_{p,r}}+\|u_0\|^2_{B^{s-1}_{p,\infty}}\|u_0\|_{F^{s+1}_{p,r}}\right)\nonumber\\
&\lesssim t\left(1+\|u_0\|^2_{F^{s-1}_{p,r}}\|u_0\|_{F^{s+1}_{p,r}}\right),
\end{align}
where we have used that $F_{p, r}^{s-1}$ is a Banach algebra with $s-1>\max\{\frac{1}{p}, \frac{1}{2}\}$ in the third inequality.

Following the same procedure of estimates as above, then by Lemma \ref{le1}, we have
\bal\label{zyl2}
\|u(t)-u_0\|_{B^{s-1}_{p,\infty}}
&\leq \int^t_0\|\pa_\tau u\|_{B^{s-1}_{p,\infty}} \dd\tau
\nonumber\\&\les \int^t_0\|\mathbf{P_1}(u)\|_{B^{s-1}_{p,\infty}} \dd\tau+\int^t_0\|\mathbf{P_2}(u)\|_{B^{s-1}_{p,\infty}} \dd\tau+ \int^t_0\|u^2 \pa_xu\|_{B^{s-1}_{p,\infty}} \dd\tau\nonumber\\&\leq \int^t_0\|(u_x)^3\|_{B^{s-2}_{p,\infty}} \dd\tau+\int^t_0\|u(u_x)^2,u^3\|_{B^{s-2}_{p,\infty}} \dd\tau+ \int^t_0\|u^2\|_{B^{s-1}_{p,\infty}}\|u\|_{B^{s}_{p,\infty}} \dd\tau
\nonumber\\&\les \int^t_0\|u_x\|_{B^{s-2}_{p,\infty}}\|(u_x)^2,uu_x\|_{B^{s-1}_{p,\infty}} \dd\tau+ \int^t_0\|u^2\|_{B^{s-1}_{p,\infty}} \|u\|_{B^{s}_{p,\infty}} \dd\tau\lesssim t\|u_0\|_{F^{s-1}_{p,r}}
\end{align}
and
\bal\label{zyl3}
\|u(t)-u_0\|_{F^{s+1}_{p,r}}
&\leq \int^t_0\|\pa_\tau u\|_{F^{s+1}_{p,r}} \dd\tau
\nonumber\\&\leq \int^t_0\|\mathbf{P_1}(u)\|_{F^{s+1}_{p,r}} \dd\tau+\int^t_0\|\mathbf{P_2}(u)\|_{F^{s+1}_{p,r}} \dd\tau+ \int^t_0\|u^2 \pa_xu\|_{F^{s+1}_{p,r}} \dd\tau\nonumber\\
&\lesssim t\left(\|u\|^3_{L_t^\infty(F^{s}_{p,r})}
+\|u\|_{L_t^\infty(L^\infty)}\|u_x^2\|_{L_t^\infty(F^{s}_{p,r})}+\|u_x\|^2_{L_t^\infty(L^\infty)}\|u\|_{L_t^\infty(F^{s}_{p,r})}\right)\nonumber\\&
\quad+ t\left(\|u^2\|_{L_t^\infty(F^{s+1}_{p,r})}\|u_x\|_{L_t^\infty(L^\infty)}
+\|u\|^2_{L_t^\infty(L^\infty)}\|u\|_{L_t^\infty(F^{s+2}_{p,r})}\right)
\nonumber\\&\lesssim t\left(1+\|u_0\|_{F^{s-1}_{p,r}}\|u_0\|_{F^{s+1}_{p,r}}+\|u_0\|^2_{F^{s-1}_{p,r}}\|u_0\|_{F^{s+2}_{p,r}}\right),
\end{align}

Next, we estimate the Triebel-Lizorkin norm for the term $u(t)-u_0+t\mathbf{v}_0$. By the Mean Value Theorem again, from \eqref{zyl1}-\eqref{zyl3} we obtain
\bbal
\|u(t)-u_0+t\mathbf{v}_0\|_{F^s_{p,r}}
&\leq \int^t_0\|\partial_\tau u+\mathbf{v}_0\|_{F^s_{p,r}} \dd\tau \\
&\leq \int^t_0\|\mathbf{P}_1(u)-\mathbf{P}_1(u_0)\|_{F^s_{p,r}} \dd\tau+\int^t_0\|\mathbf{P}_2(u)-\mathbf{P}_2(u_0)\|_{F^s_{p,r}} \dd\tau \\
&\quad+\int^t_0\|u^2\partial_xu-u^2_0\partial_xu_0\|_{F^s_{p,r}}\dd\tau\\
&\lesssim
t^{2}\left[1+\|u_0\|^2_{F^{s-1}_{p,r}}\|u_0\|_{F^{s+1}_{p,r}}
+\|u_0\|^3_{F^{s-1}_{p,r}}\Big(\|u_0\|_{F^{s+1}_{p,r}}+\|u_0\|_{F^{s-1}_{p,r}}\|u_0\|_{F^{s+2}_{p,r}}\Big)\right],
\end{align*}
where we have used
\begin{align*}
\|\mathbf{P}_1(u)-\mathbf{P}_1(u_0)\|_{F^s_{p,r}}+\|\mathbf{P}_2(u)-\mathbf{P}_2(u_0)\|_{F^s_{p,r}}&\les\|u-u_0\|_{F^s_{p,r}}
\end{align*}
and
\begin{align*}
\|u^2u_x-u^2_0\partial_xu_0\|_{F_{p, r}^s}&\les\|u^3-u^3_0\|_{F_{p, r}^{s+1}}\les\|(u-u_0)(u^2+uu_0+u^2_0)\|_{F_{p, r}^{s+1}}\\
&\les\|u-u_0\|_{B_{p, \infty}^{s-1}}\|u_0\|_{F_{p, r}^{s-1}}\|u_0\|_{F_{p, r}^{s+1}}+\|u-u_0\|_{F_{p, r}^{s+1}}\|u_0\|^2_{F_{p, r}^{s-1}}.
   \end{align*}
Thus, we complete the proof of Proposition \ref{pro 3.1}.
\end{proof}

{\bf Step2: Construction of Initial Data}

Firstly, we need to introduce smooth, radial cut-off functions to localize the frequency region. Precisely,
let $\widehat{\phi}\in \mathcal{C}^\infty_0(\mathbb{R})$ be an even, real-valued and non-negative function on $\R$ and satisfy
\begin{numcases}{\widehat{\phi}(\xi)=}
1,&if $|\xi|\leq \frac{1}{4}$,\nonumber\\
0,&if $|\xi|\geq \frac{1}{2}$.\nonumber
\end{numcases}
\begin{lemma}\label{ley2} Let $s\in\R$. Define the high-low frequency functions $f_n$ and $g_n$ by
\bbal
&f_n=2^{-ns}\phi\left(x\right)\sin \left(\frac{17}{12}2^nx\right),
\\&g_n=2^{-\frac{n}{2}}\phi\left(x\right),\quad n\gg1.
\end{align*}
Then for any $\sigma\in\R$, we have
\bbal
&\|f_n\|_{F^\sigma_{p,r}}\leq C2^{(\sigma-s)n}\|\phi\|_{L^p}\quad\text{and}\quad \|g_n\|_{F^\sigma_{p,r}}\leq C2^{-\frac{n}{2}}\|\phi\|_{L^p},\\
&\liminf_{n\rightarrow \infty}\|g_n^2\pa_xf_n\|_{F^{s}_{p,\infty}}\geq M.
\end{align*}
\end{lemma}
\begin{proof}
The proof is similar to that in \cite{L20,Lyz}. We omit the details.
\end{proof}
{\bf Step3: Difference between the data-to-solution map $\mathbf{S}_t(f_n+g_n)$ and $\mathbf{S}_t(f_n)$}

Set $u^n_0=f_n+g_n$. Obviously, we have
\bbal
\|u^n_0-f_n\|_{F^s_{p,r}}=\|g_n\|_{F^s_{p,r}}\leq C2^{-\frac{n}{2}},
\end{align*}
which means that
\bbal
\lim_{n\to\infty}\|u^n_0-f_n\|_{F^s_{p,r}}=0.
\end{align*}
It is easy to show that
\bbal
&\|f_n\|_{F^{s-1}_{p,r}}\lesssim 2^{-n},\quad\|u^n_0\|_{F^{s-1}_{p,r}}\lesssim 2^{-\frac{n}{2}},\\
&\|u^n_0,f_n\|_{F^{\sigma}_{p,r}}\leq C2^{(\sigma-s)n} \qquad  \mathrm{for} \quad \sigma\geq s,
\end{align*}
which imply
\bbal
&\|u_0^n\|^2_{F^{s-1}_{p,r}}\|u_0^n\|_{F^{s+1}_{p,r}}
+\|u_0^n\|^3_{F^{s-1}_{p,r}}\Big(\|u_0^n\|_{F^{s+1}_{p,r}}+\|u_0^n\|_{F^{s-1}_{p,r}}\|u_0^n\|_{F^{s+2}_{p,r}}\Big)\lesssim  1,
\\&\|f_n\|^2_{F^{s-1}_{p,r}}\|f_n\|_{F^{s+1}_{p,r}}
+\|f_n\|^3_{F^{s-1}_{p,r}}\Big(\|f_n\|_{F^{s+1}_{p,r}}+\|f_n\|_{F^{s-1}_{p,r}}\|f_n\|_{F^{s+2}_{p,r}}\Big)\lesssim  1.
\end{align*}
Notice that
\bbal
&\mathbf{S}_{t}(\underbrace{f_n+g_n}_{=~u^n_0})=\underbrace{\mathbf{S}_{t}(u^n_0)-u^n_0+t[(u ^n_{0})^2\pa_xu^n_{0}-\mathbf{P_1}(u^n_0)-\mathbf{P_2}(u^n_0)]}_{=~\mathbf{I}_1(u^n_0)}+f_n+g_n-t[(u ^n_{0})^2\pa_xu^n_{0}-\mathbf{P_1}(u^n_0)-\mathbf{P_2}(u^n_0)],\nonumber\\
&\mathbf{S}_{t}(f_n)=\underbrace{\mathbf{S}_{t}(f_n)-f_n+t[(f_n)^2\pa_xf_n-t\mathbf{P_1}(f_n)-t\mathbf{P_2}(f_n)]}_{=~\mathbf{I}_2(f_n)}+f_n-t[(f_n)^2\pa_xf_n-\mathbf{P_1}(f_n)-\mathbf{P_2}(f_n)],
\end{align*}
 by Proposition \ref{pro 3.1}, we deduce that
\bal\label{yyh}
\quad \ \|\mathbf{S}_{t}(u^n_0)-\mathbf{S}_{t}(f_n)\|_{F^s_{p,r}}
\geq&~\left\|\mathbf{I}_1(u^n_0)-\mathbf{I}_2(f_n)+g_n-t[(u ^n_{0})^2\pa_xu^n_{0}-(f_n)^2\pa_xf_n]\right\|_{F^s_{p,r}}\nonumber\\
&-t\left\|\mathbf{P_1}(u^n_0)-\mathbf{P_1}(f_n)\right\|_{F^s_{p,r}}-t\left\|\mathbf{P_2}(u^n_0)-\mathbf{P_2}(f_n)\right\|_{F^s_{p,r}}\nonumber\\
\geq&~ t\left\|(u ^n_{0})^2\pa_xu^n_{0}-(f_n)^2\pa_xf_n\right\|_{F^s_{p,\infty}}-C2^{-n/2}-Ct^{2}.
\end{align}
Due to $$
(u^n_{0})^2\pa_xu^n_{0}-(f_n)^2\pa_xf_n=(g_n)^2\pa_xf_n+2f_{n}g_{n}\pa_xf_{n}+(u^n_0)^2\pa_xg_n,
$$
after simple calculation, we obtain
\bbal
\big|\big|2f_{n}g_{n}\pa_xf_{n}+(u^n_0)^2\pa_xg_n\big|\big|_{F^s_{p,r}}&\leq  C2^{-n/2}
\end{align*}
and
$$
\left\|\mathbf{P_1}(u^n_0)-\mathbf{P_1}(f_n)\right\|_{F^s_{p,r}} \leq C\|u^n_0-f_n\|_{F^s_{p,r}} (\|u^n_0\|^2_{F^s_{p,r}}+\|f_n\|^2_{F^s_{p,r}})\leq  C2^{-n/2},
$$
$$
\left\|\mathbf{P_2}(u^n_0)-\mathbf{P_2}(f_n)\right\|_{F^s_{p,r}} \leq C\|u^n_0-f_n\|_{F^s_{p,r}} (\|u^n_0\|^2_{F^s_{p,r}}+\|f_n\|^2_{F^s_{p,r}})\leq  C2^{-n/2}.
$$
Using the fact
\begin{eqnarray*}
      \liminf_{n\rightarrow \infty} \|g^2_n\pa_xf_n\|_{F^s_{p,\infty}}\geq M,
        \end{eqnarray*}
from \eqref{yyh} we obtain
\bbal
\liminf_{n\rightarrow \infty}\|\mathbf{S}_t(f_n+g_n)-\mathbf{S}_t(f_n)\|_{F^s_{p,r}}\gtrsim t\quad\text{for} \ t \ \text{small enough}.
\end{align*}
This completes the proof of Theorem \ref{the1}.

\section*{Acknowledgements}
J. Li is supported by the National Natural Science Foundation of China (11801090 and 12161004) and Jiangxi Provincial Natural Science Foundation (20212BAB211004 and 20224BAB201008). Y. Yu is supported by the National Natural Science Foundation of China (12101011). W. Zhu is supported by the National Natural Science Foundation of China (12201118) and Guangdong
Basic and Applied Basic Research Foundation (2021A1515111018).

\section*{Declarations}
\noindent\textbf{Data Availability} No data was used for the research described in the article.

\vspace*{1em}
\noindent\textbf{Conflict of interest}
The authors declare that they have no conflict of interest.

\end{document}